\documentclass[12pt]{article}
\usepackage{amsmath,fullpage,amsthm}
\usepackage{amscd}
\usepackage{amsthm,mathtools} \usepackage{amssymb}
\usepackage{latexsym}
\usepackage{eufrak}
\usepackage{euscript}
\usepackage{epsfig}
\usepackage{tikz}
\usepackage{graphics}
\usepackage{array}
\usepackage{enumerate} \usepackage{authblk}

\theoremstyle{theorem}
\newtheorem{theorem}{Theorem}[section]
\theoremstyle{corollary}

\theoremstyle{lemma}

\theoremstyle{definition}

\theoremstyle{proof}
\theoremstyle{remark}

\theoremstyle{example}
\newtheorem{example}{Example}[section]
\theoremstyle{observation}

\begin{document}
  \setcounter{Maxaffil}{2}
   \title{On the Seidel spectrum  of threshold graphs}
   \author[ ]{\rm Santanu Mandal\thanks{santanu.vumath@gmail.com}}
   \author[ ]{\rm Ranjit Mehatari\thanks{ranjitmehatari@gmail.com, mehatarir@nitrkl.ac.in}}
   \affil[ ]{Department of Mathematics,}
   \affil[ ]{National Institute of Technology Rourkela,}
   \affil[ ]{Rourkela - 769008, India}
   \maketitle
 
\begin{abstract}

In this paper, we analyse spectral properties of Seidel matrix (denoted by $S$) of connected threshold graphs. We compute the characteristic polynomial and determinant of Seidel matrix of threshold graphs. We derive formulas for the multiplicity of the eigenvalues $\pm 1$ of $S$. Further we determine threshold graphs with at most 5 distinct Seidel eigenvalues. Finally we construct families of Seidel cospectral threshold graphs.
\end{abstract}
\section{Introduction}
Let $P_n,$ $C_n,$ and $K_n$ denote the path, the cycle, and the complete graph on $n$ vertices respectively. A graph with no induced subgraph isomorphic to $P_4,$ $C_4$ or $2K_2$ is called a threshold graph. Threshold graphs has various interesting applications \cite{Hammer,Henderson} and there are various equivalent definitions for them (see \cite{Peled}). Most interesting fact for threshold graph is that a threshold graph with $n$ vertices can always be represented by a finite binary string of length $n$. We construct a threshold graph by a repetitive process which start with an isolated vertex, and where at each step, either a new isolated vertex is added, or a dominating vertex is added. We represent a threshold graph $G$ on $n$ vertices using a binary string (sometimes called creation sequence of the threshold graph) $b=\alpha_1 \alpha_2 \alpha_3 \ldots \alpha_n$.   Here $\alpha_i=0$ if the vertex $v_i$ is added as an isolated vertex, and $\alpha_i =1$ if $v_i$ is added as a dominating vertex. We always take $\alpha_1=0$. Every threshold graph has a unique binary string and for each $n\geq2$, there are exactly $2^{n-2}$ distinct connected threshold graphs. For more interesting properties of  threshold graphs, we refer the book \cite{Peled}.\\

Since last decade investigation on the spectral properties of adjacency eigenvalues gained lot of attention. We found lot of papers in this direction \cite{Aguilar 1,Aguilar 2, Bapat, Ghorbani, Jacobs 1,Jacobs 2, Jacobs 3, Tura 1, Sciriha}. Bapat \cite{Bapat} proved that the number of negative, zero, and positive eigenvalues of a threshold graph can be find out directly from its binary representation. He also calculated the determinant value of the adjacency matrix. Some interesting spectral properties of threshold graphs were given by Sciriha and Farrugia in \cite{Sciriha}. Jacobs et al. wrote several papers (see \cite{Jacobs 1, Jacobs 2, Jacobs 3}) with major focus on  eigenvalue location, characteristic polynomial and energy of the adjacency matrix of threshold graphs.  Lazzarin et al. \cite{Tura 1} proved that no threshold graphs are cospectral with respect to its adjacency matrix. We found articles with focus on other matrices associated to threshold graphs. In \cite{Ranjit 1}, Banerjee  and Mehatari derived some useful results on normalized adjacency spectrum of threshold graphs, where as in \cite{Lu}, Lu et al. focused on distance spectra of threshold graphs.\\ 

In this paper we consider the Seidel matrix \cite{Brouwer,Godsil} of connected threshold graph. Let $G=(V,E)$ be a finite, undirected, simple, connected graph and let $A$ denote the adjacency matrix of $G$. Then the Seidel matrix $S$ of the graph $G$ is defined by $$S=J-I-2A.$$
In other words, if $s_{ij}$ is the $(i,j)$-th entry of $S$, then
$$s_{ij}=\begin{cases}
-1&\text{if }i\sim j,\\
1&\text{if }i\nsim j,\ i\neq j,\\
0&\text{if }i=j,
\end{cases}
$$



Let $b=\alpha_1\alpha_2\alpha_3\ldots\alpha_n$ be the binary string of a threshold graph $G$. Then the adjacency matrix $A$ of $G$ has the form
$$A=
\left[\begin{array}{cccccc}
0&\alpha_2&\alpha_3&\alpha_4&\cdots&\alpha_n\\
\alpha_2&0&\alpha_3&\alpha_4&\cdots&\alpha_n\\
\alpha_3&\alpha_3&0&\alpha_4&\cdots&\alpha_n\\
\alpha_4&\alpha_4&\alpha_4&0&\cdots&\alpha_n\\
\cdots&\cdots&\cdots&\cdots&\cdots&\cdots\\
\alpha_n&\alpha_n&\alpha_n&\alpha_n&\cdots&0
\end{array}\right]$$
Then the Seidel matrix $S$ corresponding to the threshold graph  $G$ is $$S=J-I-2A.$$
Therefore the Seidel matrix of $G$ is given by
$$S=\left[\begin{array}{cccccc}
0 & 1- 2\alpha_2 & 1- 2\alpha_3    & \ldots & 1- 2\alpha_n \\
1-2\alpha_2 &  0  & 1-2\alpha_3  & \ldots &  1-2\alpha_n \\
1-2\alpha_3 &1- 2\alpha_3 & 0 &  \ldots & 1-2\alpha_n\\
\ldots&\ldots&\ldots&\ldots&\ldots\\
1-2\alpha_n &1-2\alpha_n & 1-2\alpha_n  & \ldots & 0
\end{array}\right]$$
If we take $1-2\alpha_i=\beta_i$ for $i=1,2, 3, ..., n$, then $S$ takes the form
$$S=\left[\begin{array}{cccccc}
0 &  \beta_2 &  \beta_3    & \ldots & \beta_n \\
\beta_2 &  0  & \beta_3   & \ldots &  \beta_n \\
\beta_3 & \beta_3 & 0 &  \ldots & \beta_n\\
\ldots&\ldots&\ldots&\ldots&\ldots\\
\beta_n & \beta_n & \beta_n  & \ldots & 0
\end{array}\right],$$
where $\beta_i=1$ if $\alpha_i=0$ and $\beta_i=-1$ if $\alpha_i=1$. Thus if $b=\alpha_1\alpha_2\alpha_3\ldots\alpha_n$ is the binary string of a threshold graph then the entries of $S$ are given by,
$$s_{ij}=\begin{cases} \beta_i, &\text{for } i>j\\
              \beta_j,& \text{for } j>i\\
              0,& \text{otherwise.}
               \end{cases}$$

The whole paper is organized as follows: In section $3$ we give a recurrence formula for calculating the characteristic polynomial and determinant of a threshold graph. In section $4$ we prove some important properties of the Seidel quotient matrix $Q_S$. We prove that $Q_S$ is diagonalizable and has simple real eigenvalue. Later on, in that section, we derive the formula for multiplicity of the eigenvalues $\pm 1$ of Seidel matrix $S$. In section $5$ we derive some classes of threshold graphs with few distinct Seidel eigenvalues. We show that no threshold graph can have three distinct Seidel eigenvalues. In the last section we prove a very rare result. We show that two nonisomorphic threshold graphs may be cospectral with respect to their Seidel matrices.
\section{Determinant and characteristic polynomial}
In this section we obtain a recurrence formula for calculating the characteristic polynomial and determinant of the Seidel matrix $S$ of a threshold graphs with binary string $b=\alpha_1\alpha_2\ldots \alpha_n=0^{s_1} 1^{t_1} 0^{s_2} \ldots 0^{s_k} 1^{t_k}.$ The determinant of the Seidel matrix $S$ of  a threshold graph can be found recursively using its binary string. To obtain that, first we recall a theorem by Bapat. 
\begin{theorem}[Theorem 1, \cite{Bapat}]
Let $\alpha_2,\alpha_3,\ldots, \alpha_n$ be real numbers and $M$ be the matrix of the form 
$$M=
\left[\begin{array}{cccccc}
0&\alpha_2&\alpha_3&\alpha_4&\cdots&\alpha_n\\
\alpha_2&0&\alpha_3&\alpha_4&\cdots&\alpha_n\\
\alpha_3&\alpha_3&0&\alpha_4&\cdots&\alpha_n\\
\alpha_4&\alpha_4&\alpha_4&0&\cdots&\alpha_n\\
\cdots&\cdots&\cdots&\cdots&\cdots&\cdots\\
\alpha_n&\alpha_n&\alpha_n&\alpha_n&\cdots&0
\end{array}\right].$$
Then there exists a $n\times n$ matrix $P$ with  $det(P)=1$ such that

\begin{equation*}
PMP^T=\begin{bmatrix}
-2\alpha_2 & \alpha_2 & 0 & 0 & \ldots & 0 & 0 & 0\\ 
\alpha_2 & -2\alpha_3 & \alpha_3 & 0 & \ldots & 0 & 0 & 0\\
0 & \alpha_3 & -2\alpha_4 & \alpha_4 & \ldots & 0 & 0 & 0\\
\dots & \ldots & \ldots & \dots & \ldots & \dots & \dots & \dots\\
0 & 0 & 0 & 0 & \ldots & \alpha_{n-1} & -2\alpha_n & \alpha_n\\
0 & 0 & 0 & 0 & \ldots & 0 & \alpha_n & 0
\end{bmatrix}.
\end{equation*}
\end{theorem}
\medskip

Therefore using above theorem, we conclude that the determinant of the Seidel matrix $S$ is equal to the determinant of the following tridiagonal matrix:
\begin{equation*}
T=\begin{bmatrix}
-2\beta_2 & \beta_2 & 0 & 0 & \ldots & 0 & 0 & 0\\ 
\beta_2 & -2\beta_3 & \beta_3 & 0 & \ldots & 0 & 0 & 0\\
0 & \beta_3 & -2\beta_4 & \beta_4 & \ldots & 0 & 0 & 0\\
\dots & \ldots & \ldots & \dots & \ldots & \dots & \dots & \dots\\
0 & 0 & 0 & 0 & \ldots & \beta_{n-1} & -2\beta_n & \beta_n\\
0 & 0 & 0 & 0 & \ldots & 0 & \beta_n & 0
\end{bmatrix}
\end{equation*}

By Algorithm  2.1 of \cite{Mikkawy}, we know that the determinant of a tridiagonal matrix $T_1$, where
\begin{equation*}
T_1=
\begin{bmatrix}
b_1 & c_1 & 0 & 0 & \ldots & 0 & 0 & 0\\ 
a_2 & b_2 & c_2 & 0 & \ldots & 0 & 0 & 0\\
0 & a_3 & b_3 & c_3 & \ldots & 0 & 0 & 0\\
\dots & \ldots & \ldots & \dots & \ldots & \dots & \dots & \dots\\
0 & 0 & 0 & 0 & \ldots & a_{n-1} & \beta_{n-1} & c_{n-1}\\
0 & 0 & 0 & 0 & \ldots & 0 & a_n & b_n
\end{bmatrix},
\end{equation*}
is given by
\begin{equation}\nonumber
det(T_1)=\prod_{i=1}^{n} d_i,
\end{equation}
where
\begin{equation}\nonumber
d_i=\begin{cases} b_1, \text{ if} \  i=1, \\
                   b_i -\frac{a_i}{d_{i-1}} c_{i-1}, \text{ if} \  i=2, \  3,\  \dots,  \ n.
                 \end{cases}
\end{equation}
Now, to find the determinant of $S$, we apply above algorithm to $T$. Since $\beta_i\in\{-1,1\}$, for $i=1,2,\ldots,n$, we have\\
$d_1 =-2 \beta_2$,\\
$d_i =-2\beta_{i+1} -\frac{\beta_i ^2}{d_{i-1}}=-2\beta_{i+1} -\frac{1}{d_{i-1}}, \text{ for} \  i=2, 3, 4, \ldots,  {n-1}$. \\
$d_n =-\frac{\beta_n ^2}{d_{n-1}}=-\frac{1}{d_{n-1}}$.\\

Therefore the determinant of the Seidel matrix is given by
\begin{equation*}
det(S)=det(T) = \displaystyle{\prod_{i=1}^{n}}d_i.
\end{equation*}
\begin{example} Consider the threshold graph with the binary string $b=\alpha_1\alpha_2\alpha_3\alpha_4\alpha_5\alpha_6=001111$. Then $\beta_1=\beta_2=1,~\beta_3=\beta_4=\beta_5=\beta_6=-1$. The corresponding Seidel matrix is\\
\begin{equation}\nonumber
S=\begin{bmatrix}
0&1&-1&-1&-1&-1\\
1&0&-1&-1&-1&-1\\
-1&-1&0&-1&-1&-1\\
-1&-1&-1&0&-1&-1\\
-1&-1&-1&-1&0&-1\\
-1&-1&-1&-1&-1&0
\end{bmatrix}.
\end{equation}
Here $d_1=-2,~d_2=\frac{5}{2},~d_3=\frac{8}{5},~d_4=\frac{11}{8},~d_5=\frac{30}{11},~d_6=-\frac{11}{30}.$ \\ Therefore, $det(S)=d_1d_2d_3d_4d_5d_6=11$.
\end{example}
\begin{theorem}
\label{TH1}
Let $b=\alpha_1\alpha_2\ldots \alpha_n$ be the binary string of a threshold graph and let $b_r=\alpha_1\alpha_2\alpha_3\ldots\alpha_r$. Suppose  $\Phi_r(x)$ denote the characteristic polynomial of Seidel matrix of the threshold graph with binary string $b_r$, then the characteristic polynomial, $\Phi_n (x)$ of the Seidel matrix is obtained by the following recurrence formula
$$\Phi_r(x)=2(x+\beta_{r-1})\Phi_{r-1} (x) - 2(x+\beta_{r-1})^2 \Phi_{r-2}(x),$$ 
where $\Phi_1(x)=x$ and  $\Phi_2(x)=x^2-1$.
\end{theorem}
\begin{proof}
 Let $\Phi_r (x)$ be the characteristic polynomial of the threshold graph with binary string $b_r=\alpha_1\alpha_2\alpha_3\ldots\alpha_r$. Then

$$\Phi_r(x) =\begin{vmatrix}
x &  -\beta_2 &  -\beta_3    & \ldots & -\beta_r \\
-\beta_2 &  x  & -\beta_3  &  \ldots &  -\beta_r \\
-\beta_3 & -\beta_3 & x  & \ldots & -\beta_r\\
\ldots&\ldots&\ldots&\ldots&\ldots\\
-\beta_r & -\beta_r & -\beta_r  & \ldots & x
\end{vmatrix}.$$
We now consider the following two cases.\\


\textbf{Case I.} If $\beta_r=\beta_{r-1}$.
Then 
\begin{eqnarray*}
\Phi_r (x)&=&\begin{vmatrix}
x &  -\beta_2 &  -\beta_3   & \ldots & -\beta_{r-1} & -\beta_{r-1}\\
-\beta_2 & x & -\beta_3   & \ldots & -\beta_{r-1} & -\beta_{r-1} \\
-\beta_3 & -\beta_3 & x  & \ldots & -\beta_{r-1} &-\beta_{r-1}\\
\ldots & \ldots & \ldots&\ldots&\ldots&\ldots \\
-\beta_{r-1} & - \beta_{r-1} &  -\beta_{r-1}   & \ldots &  x &-\beta_{r-1}\\
-\beta_{r-1}&-\beta_{r-1}&-\beta_{r-1}&\ldots&-\beta_{r-1}&x
\end{vmatrix}\\
 & =&\begin{vmatrix}
-x &  \beta_2 &  \beta_3   & \ldots & \beta_{r-1} & 0\\
\beta_2 & -x & \beta_3   & \ldots & \beta_{r-1} & 0\\
\beta_3 & \beta_3 & -x  & \ldots & \beta_{r-1} & 0\\
\ldots & \ldots & \ldots&\ldots&\ldots&\ldots \\
\beta_{r-1} &  \beta_{r-1} &  \beta_{r-1}   & \ldots &  -x & -b\\
0&0&0&\ldots&-b&2b
\end{vmatrix}
\end{eqnarray*}
where $ b=\beta_{r-1}+x$. \\
Therefore,
\begin{eqnarray*}
\Phi_r(x)&&=2b \Phi_{r-1}(x)+b \begin{vmatrix}
x &  -\beta_2 &  -\beta_3   & \ldots &-\beta_{r-2}& -\beta_{r-1} \\
-\beta_2 & x & -\beta_3   & \ldots &-\beta_{r-2}& -\beta_{r-1} \\
-\beta_3 & -\beta_3 & x  & \ldots &-\beta_{r-2}& -\beta_{r-1} \\
\ldots & \ldots & \ldots&\ldots&\ldots\\
\beta_{r-2} &  \beta_{r-2} &  \beta_{r-2}   & \ldots &  x&-\beta_{r-1} \\
0&0&0&\ldots&0&-b
\end{vmatrix}\\
&&=2b\Phi_{r-1}(x) -b^2\Phi_{r-2}(x)
\end{eqnarray*}


\textbf{Case II.} If $\beta_{r-1}=-\beta_r$. Then
\begin{eqnarray*}
\Phi_r (x)&=&\begin{vmatrix}
x &  -\beta_2 &  -\beta_3   & \ldots & -\beta_{r-1} & \beta_{r-1}\\
-\beta_2 & x & -\beta_3   & \ldots & -\beta_{r-1} & \beta_{r-1} \\
-\beta_3 & \beta_3 & x  & \ldots & -\beta_{r-1} &\beta_{r-1}\\
\ldots & \ldots & \ldots&\ldots&\ldots&\ldots \\
-\beta_{r-1} &  -\beta_{r-1} &  -\beta_{r-1}   & \ldots &  x &\beta_{r-1}\\
\beta_{r-1}&\beta_{r-1}&\beta_{r-1}&\ldots&\beta_{r-1}&x
\end{vmatrix}\\
  &=&\begin{vmatrix}
x &  -\beta_2 &  -\beta_3   & \ldots & -\beta_{r-1} & 0\\
-\beta_2 & x &- \beta_3   & \ldots & -\beta_{r-1} & 0\\
-\beta_3 & -\beta_3 & x  & \ldots & -\beta_{r-1} & 0\\
\ldots & \ldots & \ldots&\ldots&\ldots&\ldots \\
-\beta_{r-1} &  -\beta_{r-1} &  -\beta_{r-1}   & \ldots &  x & b\\
0&0&0&\ldots&b&2b
\end{vmatrix}
\end{eqnarray*} 
where $ b=\beta_{r-1}+x$. \\
Therefore,
\begin{eqnarray*}
\Phi_r(x)&&=2b \Phi_{r-1}(x)-b \begin{vmatrix}
x &  -\beta_2 &  -\beta_3   & \ldots &-\beta_{r-2}& -\beta_{r-1} \\
-\beta_2 & x & -\beta_3   & \ldots &-\beta_{r-2}& -\beta_{r-1} \\
-\beta_3 & -\beta_3 & x  & \ldots &-\beta_{r-2}&- \beta_{r-1} \\
\ldots & \ldots & \ldots&\ldots&\ldots&\ldots\\
-\beta_{r-2} &  -\beta_{r-2} &  -\beta_{r-2}   & \ldots &  x&-\beta_{r-1} \\
0&0&0&\ldots&0&b
\end{vmatrix}\\
&&=2b\Phi_{r-1}(x) -b^2\Phi_{r-2}(x).
\end{eqnarray*}


Thus, combining Case I  and Case II, we have 
$$\Phi_r(x)=2(x+\beta_{r-1})\Phi_{r-1} (x) - 2(x+\beta_{r-1})^2 \Phi_{r-2}(x).$$ 
Which completes the proof.
\end{proof}



\section{Eigenvalues of threshold graphs}
In this section, first we describe some properties of the quotient matrix  corresponding  to an equitable partition of the Seidel matrix. Using these properties we establish multiplicity of  the eigenvalues $\pm1$.
\subsection{Quotient Matrix}
Let us consider a threshold graph $G$ with the binary string $b=0^{s_1} 1^{t_1} 0^{s_2} \ldots 0^{s_k} 1^{t_k}$, where $s_i,~t_i \geq 1$. Clearly $G$ has $(s+t)$ vertices, where $s=\sum s_i$ and $t=\sum t_i$. Then the  Seidel matrix $S$ of  $G$ is a square matrices of size $(s+t)$, given by

$$S=\begin{bmatrix}
(J-I)_{s_1}&-J_{s_1\times t_1}&J_{s_1\times s_2}&-J_{s_1\times t_2}&\ldots&-J_{s_1\times t_k}\\
-J_{t_1\times s_1}&(I-J)_{t_1}&J_{t_1\times s_2}&-J_{t_1\times t_2}&\ldots&-J_{t_1\times t_k}\\
J_{s_2\times s_1}&J_{s_2\times t_1}&(J-I)_{s_2}&-J_{s_2\times t_2}&\ldots&-J_{s_2\times t_k}\\
-J_{t_2\times s_1}&-J_{t_2\times t_1}&-J_{t_2\times s_2}&(I-J)_{t_2}&\ldots&-J_{t_2\times t_k}\\
\ldots&\ldots&\ldots&\ldots&\ldots&\ldots\\
-J_{t_k\times s_1}&-J_{t_k\times t_1}&-J_{t_k\times s_2}&-J_{t_k\times t_2}&\ldots&(I-J)_{t_k}
\end{bmatrix}$$ \\
where  $J_{m\times n}$ is all 1 block matrix of size $m\times n$. Clearly the diagonal blocks of $S$ are the square matrices of size $s_1\times s_1,~t_1\times t_1,~s_2\times s_2,~t_2\times t_2,~\ldots,~t_k\times t_k$. 

 Let $\pi =\{ V_{s_1},V_{t_1},V_{s_2}\ldots, V_{t_k}\}$ be an equitable partition of $G$. We denote this equitable partition as $\pi=\{C_1, C_2, C_3, \ldots,C_{2k}\}$ where $C_i=V_{s_j}, \text{if}~ i=2j-1, \text{and}~C_i=V_{t_j}, \text{if}~ i=2j $. That means $C_i$, where $i$ is odd, contains isolated vertices and $C_j$, where $j$ is even, contains dominating vertices. Therefore for the vertex partition $\pi$ of $V$, the quotient matrix  $Q_S$ of $S$ is a square matrices of size $2k$, given by

 $$Q_S=\begin{bmatrix}
s_1 -1 & -t_1 &s_2  &-t_2&s_3  & \ldots & -t_k \\
-s_1&-(t_1 -1) & s_2 &-t_2&s_3& \ldots & -t_k\\
s_1 &t_1&s_2 -1&-t_2 &s_3 & \ldots & t_k\\
-s_1&-t_1&-s_2&-(t_2 -1)&s_3&\ldots&-t_k\\
\ldots &  \ldots & \ldots&\ldots&\ldots&\ldots&\ldots \\
-s_1 &  -t_1 &-s_2 &-t_2 &-s_3 & \ldots & -(t_k -1)
\end{bmatrix}$$
We observe that all the eigenvalues of $Q_S$ are also eigenvalues of $S$. Now we provide some interesting properties of $Q_S$. We start with the diagonalizability of $Q_S$.

\begin{theorem}
Let $b=0^{s_1} 1^{t_1} 0^{s_2} \ldots 0^{s_k} 1^{t_k}$ be the binary string of a threshold graph $G$. Then $Q_S$ is diagonalizable.
\end{theorem}

\begin{proof}

Let us consider the diagonal matrix $D=diag\{ s_1,t_2,s_2, \ldots, s_k,t_k \}$ . We observe that the matrix $D^\frac{1}{2} Q_S D^{-\frac{1}{2}}$ is a symmetric matrix. Therefore $Q_S$ is similar to the symmetric matrix $D^\frac{1}{2} Q_S D^{-\frac{1}{2}}$. This implies $Q_S$ is similar to a diagonal matrix. Therefore $Q_S$ is diagonalizable.
\end{proof}

Let $\lambda$ be an eigenvalue of $Q_S$ with corresponding eigenvector $X\in\mathbb{R}^{2k}$. Let $P_{n\times {2k}}$ be the matrix whose $i$-th row is given by
$$\textbf{\emph{e}}_{\gamma_i+2}+\textbf{\emph{e}}_{\gamma_i+1}+\cdots+\textbf{\emph{e}}_{\gamma_{i}+c_i}$$
where $\gamma_i=\sum_{k=1}^{i-1}C_k$. 
Then it is easy to verify that $SP=PQ_S$. Then $S(PX)=\lambda (PX)$. Which implies that every eigenvalue of $Q_S$ is also an eigenvalue of $S$ and the eigenvector $PX$ is constant in each vertex partition. 

\begin{theorem}
\label{Threshold_quotient_th2}
Let $b=0^{s_1} 1^{t_1} 0^{s_2} \ldots 0^{s_k} 1^{t_k}$ be the binary string of a threshold graph $G$. Then 
\begin{enumerate}
\item[(i)]
$-1$ is a simple eigenvalue of $Q_S$ if $t_k=1$.
\item[(ii)] 
$-1$ is not an eigenvalue of $Q_S$ if $t_k>1$.
\end{enumerate}
\end{theorem}
\begin{proof}
Let us assume that $Q_S$ has the eigenvalue $-1$. Then there exists a non zero vector $X=
\left[\begin{array}{ccccc}
x_1&x_2&x_3&\cdots&x_{2k}
\end{array}\right]^T$ such that $Q_S X=-X$ which gives the following system of linear equations:
\begin{eqnarray*}
(s_1 -1)x_1-t_1x_2+s_2x_3 -t_2x_4+s_3x_5-t_3x_6+\ldots-t_kx_{2k}=-x_1~~~~~&&(1)\\
-s_1x_1-(t_1-1)x_2+s_2x_3-t_2x_4+s_3x_5-t_3x_6+\ldots -t_kx_{2k}=-x_2~~~~~&&(2)\\
s_1x_1+t_1x_2+(s_2-1)x_3-t_2x_4+s_3x_5-t_3x_6+\ldots -t_kx_{2k}=-x_3~~~~~&&(3)\\
-s_1x_1-t_1x_2-s_2x_3-(t_2-1)x_4+s_3x_5-t_3x_6+\ldots -t_kx_{2k}=-x_4~~~~~&&(4)\\
s_1x_1+t_1x_2+s_2x_3+t_2x_4+(s_3-1)x_5-t_3x_6+\ldots -t_kx_{2k}=-x_5~~~~~&&(5)\\
\ldots~~\ldots~~\ldots~~\ldots~~\ldots~~~~~~~\ldots~~\ldots~~\ldots~~\ldots~~~~~~~~~~&&~~\vdots\\
-s_1x_1-t_1x_2-s_2x_3-t_2x_4-s_3x_5-t_3x_6+\ldots -(t_k-1)x_{2k}=-x_{2k}~~~&&(2k)\\
\end{eqnarray*}

Now, applying the following operations in order, we get the values of $x_i$,
for all $i=1,2,3,\ldots ,2k$. \\
$(1)-(3)$ gives $x_2=0$. Now\\
Putting $x_2=0$ and performing  $(1)-(2)$ we get $x_1=0$.\\
Putting $x_1=x_2=0$ and performing  $(1)-(5)$ we get $x_4=0$.\\
Putting $x_1=x_2=x_4=0$ and performing  $(1)-(4)$ we get $x_3=0$.\\
Proceeding in this way, and after performing $(1)-(2k-1)$ and $(1)-(2k-2)$ we obtain 
that any vector that satisfies eigenvalue equation corresponding to $-1$ must have first $2k-2$ entry equal to $0$.

Finally performing $(1)-(2k)$ and $(1)+(2k)$ we get, 
\begin{eqnarray*}
x_{2k}-s_kx_{2k-1}=0~~~~~&&(a)\\
(1-t_k)x_{2k}=0~~~~~&&(b)
\end{eqnarray*}
We now consider two cases:\\

\textbf{Case I. } If $t_k=1$. Then from (a) and (b), we get $X=\left[\begin{array}{ccccccc}
0&0&0&\cdots&0&1&s_k
\end{array}\right]^T$ is an eigenvector corresponding to $-1$. In fact, in that case any nonzero eigenvector corresponding $-1$ is a nonzero multiple of $X$. Therefore $-1$ is a simple eigenvalue of $Q_S$.\\

\textbf{Case II. }
If $ t_k \neq 1 $. Then from (b) $x_{2k}=0$. Therefore $x_{2k-1}=0$ by (a). \\
Therefore $X=\left[\begin{array}{ccccc}
x_1&x_2&x_3&\cdots&x_{2k}
\end{array}\right]^T=\left[\begin{array}{ccccc}
0&0&0&\cdots&0
\end{array}\right]^T$. \\ Thus $-1$ can not be an eigenvalue of $Q_S$ if $t_k>1$.
\end{proof}

\begin{theorem}
\label{Threshold_quotient_th3}
Let $b=0^{s_1} 1^{t_1} 0^{s_2} \ldots 0^{s_k} 1^{t_k}$ be the binary string of a threshold graph $G$. Then 
\begin{enumerate}
\item[(i)]
$1$ is a simple eigenvalue of $Q_S$ if $s_1=1$.
\item[(ii)] 
$1$ is not an eigenvalue of $Q_S$ if $s_1>1$.
\end{enumerate}
\end{theorem}
\begin{proof}
For $X\in \mathbb{R}^{2k}$ consider the matrix equation $Q_S X=X$. Which gives the following system of linear equations:
\begin{eqnarray*}
(s_1 -1)x_1-t_1x_2+s_2x_3 -t_2x_4+\ldots-t_{k-1}x_{2k-2}+s_kx_{2k-1}-t_kx_{2k}=x_1~~~~~~~&&(1)\\
-s_1x_1-(t_1-1)x_2+s_2x_3-t_2x_4+\ldots -t_{k-1}x_{2k-2}+s_kx_{2k-1}-t_kx_{2k}=x_2~~~~~~~&&(2)\\
s_1x_1+t_1x_2+(s_2-1)x_3-t_2x_4+\ldots-t_{k-1}x_{2k-2}+s_kx_{2k-1} -t_kx_{2k}=x_3~~~~~~~&&(3)\\
\ldots~~\ldots~~\ldots~~\ldots~~\ldots~~~~~~~\ldots~~\ldots~~\ldots~~\ldots~~~~~~~~~~&&~~\vdots\\
-s_1x_1-t_1x_2-s_2x_3-t_2x_4-\cdots-(t_{k-1}-1)x_{2k-2}+s_kx_{2k-1}\cdots-t_kx_{2k}=x_{2k-2}~~~&&(2k-2)\\
s_1x_1+t_1x_2+s_2x_3+t_2x_4+\cdots+t_{k-1}x_{2k-2}+(s_k-1)x_{2k-1}-t_kx_{2k}=x_{2k-1}~~~&&(2k-1)\\
-s_1x_1-t_1x_2-s_2x_3-t_2x_4-\cdots-t_{k-1}x_{2k-2}-s_kx_{2k-1}-(t_k-1)x_{2k}=x_{2k}~~~~~&&(2k)\\
\end{eqnarray*}

Now, to find the $x_i$'s, we apply the following operations:\\
$(2k)-(2k-2)$ gives $x_{2k-1}=0$.\\ 
Now Putting $x_{2k-1}=0$ and performing  $(2k)-(2k-1)$ we get $x_{2k}=0$.\\
Putting $x_{2k}=x_{2k-1}=0$ and performing  $(2k)-(2k-4)$ we get $x_{2k-3}=0$.\\
Putting $x_{2k}=x_{2k-1}=x_{2k-3}=0$ and performing  $(2k)-(2k-3)$ we get $x_{2k-2}=0$.\\
Proceeding in this way, and after performing $(2k)-(2)$ and $(2k)-(3)$ we obtain 
that for any vector $X=(x_1~x_2~x_3~\ldots~x_{2k})^t$ which satisfies $Q_sX=X$ we have
$x_3=x_4=\cdots=x_{2k}=0$
Finally performing $(2k)-(1)$ and $(2k)+(1)$ we get, 
\begin{eqnarray*}
(s_1-1)x_1=0~~~~~&&(c)\\
t_1x_2+x_1=0~~~~~&&(d)
\end{eqnarray*}
We now consider two cases:\\

\textbf{Case I. } If $s_1=1$. Then from (c) and (d), we get $X=\left[\begin{array}{ccccccc}
t_1&-1&0&0&\cdots&0
\end{array}\right]^T$ is an eigenvector corresponding to $1$. In fact, in that case any nonzero eigenvector corresponding $1$ is a nonzero multiple of $X$. Therefore $1$ is a simple eigenvalue of $Q_S$.\\

\textbf{Case II. }
If $ s_1 \neq 1 $. Then from (c) $x_{1}=0$. Therefore $x_{2}=0$ by (b). \\
Therefore $X=\left[\begin{array}{ccccc}
x_1&x_2&x_3&\cdots&x_{2k}
\end{array}\right]^T=\left[\begin{array}{ccccc}
0&0&0&\cdots&0
\end{array}\right]^T$. \\ Thus $1$ is not an eigenvalue  of $Q_S$ if $t_k>1$.
\end{proof}
From Theorem \ref{Threshold_quotient_th2} and Theorem \ref{Threshold_quotient_th3} it is clear that $\pm1$ can be an eigenvalue of $Q_S$ with multiplicity at most 1. In the next theorem we prove that $Q_S$ has $2k$ distinct eigenvalues.


\begin{theorem}
\label{Threshold_quotient_th4}
All eigenvalues of  $Q_S$ are simple.
\end{theorem}
\begin{proof}
Suppose $\lambda$ is an eigenvalue of $Q_S$. Let $X=\left[\begin{array}{ccccc}
x_1&x_2&x_3&\cdots&x_{2k}
\end{array}\right]^T$ be an eigenvector corresponding to $\lambda$ such that $x_l\neq 0$ and $x_m=0$ for all $m<l$, where $l$ is minimal. Then $l=2p-1, 1\leq p\leq k$. We already proved that $\lambda=\pm 1$ can at most be a  simple eigenvalue. Now we prove the theorem for $\lambda \neq \pm 1$.
Then from the relation $Q_SX=\lambda X$, we have the following system of linear equations.
\begin{eqnarray*}
(s_1 -1)x_1-t_1x_2+s_2x_3 -t_2x_4+\ldots-t_kx_{2k}=\lambda x_1~~~~~&&(1)\\
-s_1x_1-(t_1-1)x_2+s_2x_3-t_2x_4+\ldots -t_kx_{2k}=\lambda x_2~~~~~&&(2)\\
\ldots~~\ldots~~\ldots~~\ldots~~\ldots~~\ldots~~\ldots~~\ldots~~\ldots~~\ldots~~\ldots&&\vdots\\
-s_1x_1-t_1x_2-s_2x_3-(t_2-1)x_4+\ldots -t_kx_{2k}=\lambda x_{l-1}~~~&&(l-1)\\
s_1x_1+t_1x_2+s_2x_3+t_2x_4+\ldots -t_kx_{2k}=\lambda x_l~~~~~~&&(l)\\
-s_1x_1-t_1x_2-s_2x_3-(t_2-1)x_4+\ldots -t_kx_{2k}=\lambda x_{l+1}~~~&&(l+1)\\
s_1x_1+t_1x_2+s_2x_3+t_2x_4+\ldots -t_kx_{2k}=\lambda x_{l+2}~~~&&(l+2)\\
\ldots~~\ldots~~\ldots~~\ldots~~\ldots~~\ldots~~\ldots~~\ldots~~\ldots~~\ldots~~\ldots&&\vdots\\
-s_1x_1-t_1x_2-s_2x_3-t_2x_4+\ldots -(t_k-1)x_{2k}=-x_{2k}~~~&&(2k)
\end{eqnarray*}

Now, applying the following operations in order, we get the values of $x_i$ for all $i=l, l+1, l+2, \ldots ,2k$. \\
$(l)-(l+1)$ gives $$x_{l+1}=\dfrac{(1+\lambda -2s_{\frac{l+1}{2}})x_l}{\lambda- 1}=c_{l+1}x_l, \text{ (say)}.$$\\
Again $(l)-(l+2)$ gives,
 $$x_{l+2}=\frac{1}{\lambda+1}\Big[2t_\frac{l+1}{2}x_{l+1}+(1+\lambda) x_l\Big] =c_{l+2}x_l,\text{ (say)}.$$
and so on.\\
Thus, proceeding in this way, we get the constants $c_{l+1}, c_{l+2}, c_{l+3}, \ldots, c_{2k}$, such that,
$$X=x_l\left[\begin{array}{ccccccccc}
0&0&\cdots&0&1&c_{l+1}&c_{l+2}&\cdots&c_{2k}
\end{array}\right]^T.$$ Now if $X'=\left[\begin{array}{ccccc}
x_1'&x_2'&x_3'&\cdots&x_{2k}'
\end{array}\right]^T$ be the another eigenvector corresponding to $\lambda$, then we see that $X^{'}$ is a constant multiple of $X$. Hence the geometric multiplicity of $\lambda$ is one. Again $Q_S$ is diagonalizable. Hence algebraic multiplicity of $\lambda$ is also one.
Hence all eigenvalues of $Q_S$ are simple.
\end{proof}


\subsection{Multiplicity of the eigenvalues $\pm$ 1}
Let us consider a threshold graph $G$ with the binary string $0^{s_1} 1^{t_1} 0^{s_2} \ldots 0^{s_k} 1^{t_k}$. Let $n_{-1}(S)$ and $n_{+1}(S)$ denote the multiplicity of the eigenvalues $-1$ and $+1$  respectively of the Seidel matrix $S$.  We now derive formulas for  $n_{-1}(S)$ and $n_{+1}(S)$. For that first we construct eigenvectors corresponding to $\pm 1$ which does not belong to spectrum of $Q_S$. Now $S$ has the form
$$S=\begin{bmatrix}
(J-I)_{s_1}&-J_{s_1\times t_1}&J_{s_1\times s_2}&-J_{s_1\times t_2}&\ldots&-J_{s_1\times t_k}\\
-J_{t_1\times s_1}&(I-J)_{t_1}&J_{t_1\times s_2}&-J_{t_1\times t_2}&\ldots&-J_{t_1\times t_k}\\
J_{s_2\times s_1}&J_{s_2\times t_1}&(J-I)_{s_2}&-J_{s_2\times t_2}&\ldots&-J_{s_2\times t_k}\\
-J_{t_2\times s_1}&-J_{t_2\times t_1}&-J_{t_2\times s_2}&(I-J)_{t_2}&\ldots&-J_{t_2\times t_k}\\
\ldots&\ldots&\ldots&\ldots&\ldots&\ldots\\
-J_{t_k\times s_1}&-J_{t_k\times t_1}&-J_{t_k\times s_2}&-J_{t_k\times t_2}&\ldots&(I-J)_{t_k}
\end{bmatrix}.$$
For $i>1$ define set of $\{i-1\}$ orthogonal row-vectors $\{X_j^i\}$ in $\mathbb{R}^i$ by
$$X_j^i=\textbf{\emph{e}}_1(i)^T+\textbf{\emph{e}}_2(i)^T+\cdots+\textbf{\emph{e}}_{j}(i)^T-j\textbf{\emph{e}}_{j+1}(i)^T\ \forall 1\leq j\leq i-1,$$
where $\textbf{\emph{e}}_{j}(i)$ is the $j$-th standard basis element of $\mathbb{R}^i$. \\

Now for $s_i\geq 2$, define
$$Y_{s_i}(j)=[O_{s_1}\ O_{t_1}\ \cdots\ O_{t_{i-1}}\ X_j^{s_i} \ O_{t_{i}}\ \cdots\ O_{t_k}]^T\ \ \forall 1\leq i\leq k, 1\leq j \leq s_i-1,$$
where $O_r$ denote the $r$-component zero row-vector.
The the set $\{Y_{s_i}(1),Y_{s_i}(2),\ldots, Y_{s_i}(s_i-1)\}$ contains $s_i-1$ orthogonal eigenvectors corresponding to -1.\\

Again for each $t_i\geq 2$, define
$$Z_{t_i}(j)=[O_{s_1}\ O_{t_1}\ \cdots\ O_{s_{i}}\ X_j^{t_i} \ O_{s_{i+1}}\ \cdots\ O_{t_k}]^T\ \ \forall 1\leq i\leq k, 1\leq j \leq t_i-1,$$
The the set $\{Z_{t_i}(1),Z_{t_i}(2),\ldots, Z_{t_i}(t_i-1)\}$ contains $t_i-1$ orthogonal eigenvectors corresponding to 1.\\

Each of $Y_{s_i}(j)$'s and $Z_{t_i}(j)$'s has row sum zero in each of the vertex partition. Now let $\lambda$ be an eigenvalue of $Q_S$ with eigenvector $X\in \mathbb{R}^{2k}$. Then the eigenvector $PX$ corresponding to the eigenvalue $\lambda$ is constant in each vertex partition. Therefore $PX$ is orthogonal to each of these $Y_{s_i}(j)$'s and $Z_{t_i}(j)$'s. Using this fact we now calculate the multiplicity of the eigenvalue $\pm1$.

\begin{theorem}
\label{Threshold_pm1_Th1}
Let $0^{s_1} 1^{t_1} 0^{s_2} \ldots 0^{s_k} 1^{t_k}$ be the binary string of a threshold graph $G$. Then
$$n_{-1}(S)=\begin{cases} \sum s_i -k, ~\text{for}~ t_k>1\\
              \sum s_i -k +1,~ \text{for}~t_k=1.\\
             \end{cases}$$
\end{theorem}
\begin{proof}
We already observed that, if $s_i\geq 2$ then the set $\{Y_{s_i}(1),Y_{s_i}(2),\ldots, Y_{s_i}(s_i-1)\}$ contains $s_i-1$ orthogonal eigenvectors corresponding to -1. Now for $s_l, \ s_m\geq2$, the vectors $Y_{s_l}(j)$ and $Y_{s_m}(k)$ are orthogonal for all $1\leq j< m$  and $1\leq k<m$. Therefore the set 
$$\{Y_{s_i}(j)|1\leq i\leq k, 1\leq j \leq s_i-1 \text{ and }s_i>1\}$$ provides a set of $\sum s_i-k$ orthogonal eigenvectors corresponding to -1. We now consider two cases.\\
\textbf{Case  I.} Let us take $t_k>1$. Then $Q_S$ does not contain the eigenvalue $-1$ . Therefore the  multiplicity of the eigenvalue $-1$ in  $S$ is  exactly $\sum (s_i-1)$. Thus
$$n_{-1}(S)=\sum s_i -k $$
\textbf{Case II.} Let us take $t_k=1$. Then the quotient matrix $(Q_S)$ has eigenvalue $-1$ with multiplicity 1. Thus
$$n_{-1}(S)=\sum s_i -k+1 .$$
This completes proof.
\end{proof}
\begin{theorem}
\label{Thresold_pm1_Th2}
Let $0^{s_1} 1^{t_1} 0^{s_2} \ldots 0^{s_k} 1^{t_k}$ be the binary string of a threshold graph $G$. Then
$$n_{+1}(S)=\begin{cases} \sum t_i -k, ~\text{for}~ s_1>1\\
              \sum t_i -k +1,~ \text{for}~s_1=1\\
             \end{cases}$$
\end{theorem}
\begin{proof}
By a similar argument to previous theorem, the set 
$$\{Z_{t_i}(j)|1\leq i\leq k, 1\leq j \leq t_i-1 \text{ and }t_i>1\}$$ provides a set of $\sum t_i-k$ orthogonal eigenvectors corresponding to 1. We now consider two cases.\\
\textbf{Case  I.} If $s_1>1$. Then spectrum of $Q_S$ does not contain the eigenvalue $1$. Therefore the  multiplicity of the eigenvalue $1$ in  $S$ is  exactly $\sum (t_i-1)$.\\Thus
$$n_{+1}(S)=\sum t_i -k $$
\textbf{Case II.} If $s_1=1$. Then $1$ is a simple eigenvalue of $Q_S$. Thus\\
$$n_{+1}(S)=\sum t_i -k+1 $$
This completes proof.
\end{proof}

By using Theorem \ref{Threshold_quotient_th2}, \ref{Threshold_quotient_th3}, \ref{Threshold_pm1_Th1} and \ref{Thresold_pm1_Th2} we can easily calculate eigenvectors corresponding to $\pm 1$. For example, consider the threshold graph $G$ with binary string $b=01100111$. Then $n_{-1}(G)=1$ and $n_{+1}(G)=4$. An eigenvector corresponding to $-1$ is 
$$Y_{s_2}(1)=\left[\begin{array}{cccccccc}
 0&0&0&1&-1&0&0&0
\end{array}\right]^T.$$ 
Where as eigenvectors corresponding to $1$ are 
$$Z_{t_1}(1)=\left[\begin{array}{cccccccc}
 0&1&-1&0&0&0&0&0
\end{array}\right]^T$$
$$Z_{t_2}(1)=\left[\begin{array}{cccccccc}
 0&0&0&0&0&1&-1&0
\end{array}\right]^T$$
$$Z_{t_2}(2)=\left[\begin{array}{cccccccc}
 0&0&0&0&0&1&1&-2
\end{array}\right]^T$$
and $$PX=\left[\begin{array}{cccccccc}
 2&-1&-1&0&0&0&0&0
\end{array}\right]^T,$$
where $X=\left[\begin{array}{cccc}
 2&-1&0&0
\end{array}\right]^T$ is an eigenvector corresponding to $1$ for the quotient matrix $Q_S$.


\section{Threshold graph with few distinct Seidel eigenvalue}
Now we characterize classes of threshold graphs which have few distinct Seidel eigenvalues. In particular we classify all such threshold graphs which have at most five distinct Seidel eigenvalues. 


\begin{theorem}
Let $G$ be a threshold graph with binary string $b$. Then $G$ has two distinct Seidel eigenvalues if and only if  $G$ is either the complete graph $K_n$ or the star graph $S_n$.
\end{theorem}
\begin{proof}
Suppose $G$ has two eigenvalues then the binary string of $G$ is of the form $0^s1^{n-s}$. Now if $1<s<n-1$ then $\pm 1$ are eigenvalues of $S$. Again in that case 
$$Q_S=\left[\begin{array}{cc}
s-1&s-n\\
-s&s-n+1
\end{array}\right].
 $$
Then $Q_S$ has two distinct eigenvalues other than $\pm1$. Hence $S$ has  distinct eigenvalues. Therefore either $s=1$ or $s=n-1$, and in both these cases $S$ has two distinct eigenvalues.
Now, if $b=01^{n-1}$ then $G=K_n$ and if $b=0^{n-1}1$ then $G=S_n$.\\
Which completes the proof of the theorem.
\end{proof}
For a threshold graph $G$ with binary string $b=0^{s_1} 1^{t_1} 0^{s_2} \ldots 0^{s_k} 1^{t_k}$, the Seidel matrix $S$ has atleast $2k$ distinct eigenvalues and atmost $2k+2$ eigenvalues. We already observed that for a threshold graph with binary string $0^s1^{n-s}$, $1<s<n-1$, the Seidel matrix $S$ has 4 distinct eigenvalues. Again by previous theorem if $b=01^{n-1}$ or $b=0^{n-1}1$, then $S$ has exactly two distinct eigenvalues. Thus we have the following conclusion.
\begin{theorem}
No threshold graph can have three distinct Seidel eigenvalues.
\end{theorem}
In the next two theorems we characterize threshold graphs with exactly 4 or 5 distinct eigenvalues.
\begin{theorem}
Let $G$ be a threshold graph with binary string $b$. Then $G$ has 4 distinct Seidel eigenvalues if and only if  $b=01^{t_1}0^{s_2}1$ or  $b=0^{s_1}1^{t_1},~s_1>1,~t_1>1.$
\end{theorem}
\begin{proof}
Let $b=0^{s_1} 1^{t_1} 0^{s_2} \ldots 0^{s_k} 1^{t_k}$ be the binary string of the threshold graph $G$. If $S$ has 4 distinct eigenvalues then $k\leq 2.$
Now we consider following two cases.

\textbf{Case I.} 
Let $k=1$. Then $G$ is a threshold graph with binary string $b=0^{s_1}1^{t_1}$. If either of $s_1$ or $t_1$ is equal to 1, then $S$ has exactly two distinct eigenvalues. Where as if $s_1>1,t_1>1$, then the quotient matrix $Q_S$ has two distinct eigenvalues other that $\pm 1$. Therefore $S$ has four distinct eigenvalues $-1$ with multiplicity $s_1-1$ and $1$ with multiplicity $t_1-1$, and another two simple eigenvalues come from $Q_S$.

\textbf{Case II.} Let $k=2$ then $b=0^{s_1}1^{t_1}0^{s_2}1^{t_2}$. Now if $G$ has 4 distinct eigenvalues, then $S$ can not have eigenvalue $\pm 1$ outside of the spectrum of $Q_S$. Therefore 4 distinct eigenvalues is possible only if $s_1=1=t_2$, and in that case eigenvalues of $S$ are $1^{t_1}$, $(-1)^{s_2}$ and $\frac{(s_2-t_1) \pm \sqrt{(s_2 - t_1)^2 +4(1+t_1+s_2+2t_1s_2)}}{2}$.
 
Conversely if $b=01^{t_1}0^{s_2}1$ or  $b=0^{s_1}1^{t_1},~s_1>1,~t_1>1$, then $S$ has four distinct eigenvalues.
This completes the proof of the theorem.
\end{proof}

\begin{theorem}
Let $G$ be a threshold graph with binary string $b$ . Then $G$ has 5 distinct Seidel eigenvalues if and only if $b=01^{t_1}0^{s_2}1^{t_2}$ with $t_2>1$ or $b=0^{s_1}1^{t_1}0^{s_2}1$ with $s_1>1$.
\end{theorem}
\begin{proof}
Let $G$ be threshold graph with binary string $0^{s_1}1^{t_1}0^{s_2}1^{t_2} \ldots 0^{s_k}1^{t_k}$. Now for 5 distinct eigenvalues of $S$, $k$ must be equal to $2$. Let $G$ be the threshold graph with binary string $b=0^{s_1}1^{t_1}0^{s_2}1^{t_2}$. Then $Q_S$  has four distinct eigenvalues. Now $S$ will have five distinct Seidel eigenvalues if $\pm1$ are eigenvalues of $S$ and exactly one of $\pm 1$ belongs to spectrum of $Q_S$. Thus we have the following two cases.\\
\textbf{Case I.} $s_1=1, t_2>1$. Then $+1$ belongs to the spectrum of $Q_S$, where as $-1$ is not an eigenvalue of $Q_S$. Therefore spectrum of $S$ is $\{ -1^{s_2-1}, 1^{t_1+t_2}, \alpha_1, \beta_1, \gamma_1 \}$, where $\alpha_1$, $\beta_1$, and $\gamma_1$ are the distinct eigenvalues (other than $1$) of $Q_S$. Thus $S$ has five distinct eigenvalues.\\
\textbf{Case II.} $s_1>1, t_2=1$. Then $-1$ belongs to the spectrum of $Q_S$, where as $1$ is not an eigenvalue of $Q_S$. Therefore spectrum of $S$ is $\{ -1^{s_1+s_2}, 1^{t_1-1}, \alpha_2, \beta_2, \gamma_2 \}$, where $\alpha_2$, $\beta_2$, and $\gamma_2$ are the distinct eigenvalues (other than $-1$) of $Q_S$. Thus $S$ has five distinct eigenvalues.\\
This completes the proof.
\end{proof}

\section{Two threshold graphs may be Seidel cospectral}
We conclude this paper by showing that two nonisomorphic threshold graphs may be cospectral with respect to its Seidel matrix. Although it is well known that two non-isomorphic threshold graphs are not cospectral with respect to its adjacency matrix and Laplacian matrix;  but here we see that two threshold graphs with distinct binary strings may be Seidel cospectral. Using the following theorem we can construct nonisomorphic cospectral threshold graphs.
\begin{theorem}
\label{Threshold_cospectral}
Let us consider two threshold graphs $G_1$ and $G_2$ on $n$ vertices with the binary string $b_1= 0^{n-2}1^2$ and $b_2=010^{n-3}1$ respectively. Then $G_1$ and $G_2$ are always Seidel cospectral.
\end{theorem}

\begin{proof}
For the binary string $b_1=0^{n-2}1^2$, $Q_S$ is given by\\
$$Q_S=\begin{bmatrix}
n-3&-2\\
2-n&-1
\end{bmatrix}.$$
The characteristic equation of $Q_S$ is: $x^2+(4-n)x+(7-3n)=0$\\
Now for the binary string $b_2=010^{n-3}1$, $Q_S$ is given by\\
$$Q_S=\begin{bmatrix}
0&-1&n-3&-1\\
-1&0&n-3&-1\\
1&1&n-4&-1\\
-1&-1&3-n&0
\end{bmatrix}.$$
The characteristic equation of $Q_S$ is: $x^4+(4-n)x^3+(6-3n)x^2+(n-4)x+(3n-7)=0$\\
Which gives $(x+1)(x-1)[x^2+(4-n)x+(7-3n)]=0$\\
Thus for both the strings the Seidel spectrum is same which is $\{-1^{n-3}, 1, \alpha, \beta \}$, where $\alpha, \beta$ are the roots of the equation $x^2+(4-n)x+(7-3n)=0$.\\
Hence two threshold graphs $G_1$ and $G_2$ are Seidel cospectral.
\end{proof}
\begin{example}
If we take $n=4$ in Theorem \ref{Threshold_cospectral}, we get threshold graphs $G_1$ and $G_2$ with binary string $b_1=0011$ and $b_2=0101$ respectively. In that case, $G_1$ and $G_2$ are not isomorphic (see Figure \ref{Threshold_fig 1}) but they both have eigenvalues $\pm1,\pm\sqrt{5}.$
\end{example}
\begin{figure}[h]
        \centering
        \includegraphics[height=3.5cm]{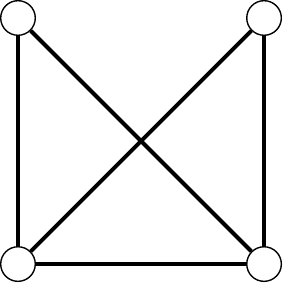}
        \hspace*{3cm}
                \includegraphics[height=3.5cm]{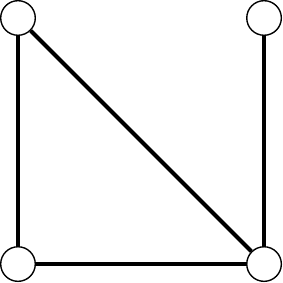}
        \caption{Non isomorphic cospectral threshold graphs with 4 vertices.}
        \label{Threshold_fig 1}
    \end{figure}
\section{Acknowledgement}
The author Santanu Mandal thanks to University Grants Commission, India for financial support under the beneficiary code BININ01569755.


\end{document}